\documentclass[12pt,a4paper]{article}
\usepackage[utf8x]{inputenc}
\usepackage{amsmath}
\usepackage{amsfonts}
\usepackage{amssymb}
\usepackage{amsthm}
\usepackage{mathrsfs}
\usepackage{mathtools}
\allowdisplaybreaks 
\usepackage{fancyhdr}
\usepackage{graphicx}
\usepackage[usenames,x11names]{xcolor}
\usepackage{framed}
\usepackage[top=2cm,bottom=2cm,margin=2cm]{geometry}
\usepackage{cancel}
\usepackage{array}
\usepackage{comment}

\usepackage{hyperref}
\usepackage[numbers,sort&compress]{natbib}

\title{A measure of intelligence of an approximation of a real number in a given model}
\author{\sc Bakir FARHI \\
National Higher School of Mathematics \\
P.O.Box 75, Mahelma 16093, Sidi Abdellah (Algiers) \\
Algeria \\[1mm]
\href{mailto:bakir.farhi@nhsm.edu.dz}{\tt bakir.farhi@nhsm.edu.dz} \\[1mm]
\url{http://farhi.bakir.free.fr/}
}
\date{}

\let\epsilon=\varepsilon

\def\R{{\mathbb R}}
\def\Q{{\mathbb Q}}

\def\N{{\mathbb N}}
\def\Z{{\mathbb Z}}

\def\s{{\mathbf{s}}}     %%%% pour size
\def\M{\mathfrak{M}}

\def\idem{\leavevmode\hbox to 10.6mm{\vrule height .63ex depth -.59ex
    width 10mm\hfill}}

\theoremstyle{plain}
\numberwithin{equation}{section}
\newtheorem{thm}{Theorem}[section]

\newtheorem{prop}[thm]{Proposition}
\newtheorem{coll}[thm]{Corollary}

\theoremstyle{definition}
\newtheorem{defi}[thm]{Definition}

\theoremstyle{remark}
\newtheorem{rmk}[thm]{Remark}
\newtheorem{rmks}[thm]{Remarks}

\newtheorem{expls}[thm]{Examples}

\begin{document}
\maketitle

\begin{abstract}
In this paper, we introduce a way to measure the intelligence (or relevance) of an approximation of a given real number in a given model of approximation. Based on the notion of complexity of a number, defined as the number of its digits (in a given base), we introduce a function noted $\mu$ (called a measure of intelligence) that associates to any approximation $\mathbf{app}$ of a given real number in a given model a positive number $\mu(\mathbf{app})$, which measures the quality of that approximation. More precisely, an approximation $\mathbf{app}$ is deemed intelligent if and only if $\mu(\mathbf{app}) \geq 1$. We illustrate the theory with several numerical examples and apply it to the rational model. In this case, we show that it is consistent with the classical theory of rational Diophantine approximation. We conclude by stating an open problem, namely whether any real number can be intelligently approximated in a given model for which it is a limit point.    
\end{abstract}

\section{Introduction}
Throughout this paper, $\N$ denotes the set of positive integers and $\N_0 := \N \cup \{0\}$ the set of non-negative integers. We denote by $\log$ the natural logarithm function. On the other hand, for convenience, we use the standard notation for regular continued fraction: we denote a finite regular continued fraction 
$$
a_0 + \dfrac{1}{a_1 + \dfrac{1}{a_2 + {}_{\ddots + \dfrac{1}{a_n}}}}
$$ 
by $[a_0 ; a_1 , a_2 , \dots , a_n]$ and we denote an infinite regular continued fraction
$$
a_0 + \dfrac{1}{a_1 + \dfrac{1}{a_2 + {}_{\ddots}}}
$$
by $[a_0 ; a_1 , a_2 , \dots]$ (where $a_0 , a_1 , \dots \in \R$). If an infinite regular continued fraction $[a_0 ; a_1 , a_2 , \dots]$ is eventually periodic, precisely if there exist $k \in \N_0$ and $d \in \N$ such that for any $n \geq k$, we have $a_{n + d} = a_n$, then we denote it (as usually) by
$$
[a_0 ; a_1 , a_2 , \dots , a_{k - 1} , \overline{a_k , a_{k + 1} , \dots , a_{k + d - 1}}] .
$$ 

More than two thousand years ago, mathematicians and calculators were interested in approximating some important real numbers by the values of some type of functions with integer variables. For instance, well-known approximations of $\pi$ include $\frac{22}{7}$ (due to Archimedes), $\sqrt{2} + \sqrt{3}$, and $\sqrt{\frac{40}{3} - 2 \sqrt{3}}$ (given by Kocha\'nski \cite{koch}). All these approximations are recognized as \emph{interesting} (or \emph{intelligent}) yet, to date, no rigorous mathematical meaning has been assigned to the word ``interesting''! For example, our intuition tells us that the approximation $\pi \approx \frac{22}{7}$ is more interesting (or more intelligent) than $\pi \approx \frac{314159}{100000}$ even though the latter approximation is more accurate than the first one. In this particular example, this preference can be explained by the fact that the size of the approximation $\frac{22}{7}$ of $\pi$ is much smaller than that of $\frac{314159}{100000}$ (where the size of a rational number $\frac{a}{b}$, with $a , b \in \Z$, $b \neq 0$ and $\mathrm{gcd}(a , b) = 1$, can be taken equal to $\max(|a| , |b|)$). A clearer explanation is that there is no rational approximation of $\pi$ that is better than $\frac{22}{7}$ with a positive denominator not exceeding $7$. An additional reason is that $\frac{22}{7}$ is a convergent of the regular continued fraction expansion of $\pi$. However, the approximation $\pi \approx \frac{314159}{100000}$ is naive, since a better rational approximation of $\pi$ exists with a smaller denominator (for example $\pi \approx \frac{355}{113}$). From this example, we see that the \emph{intelligence} (or the interest) of an approximation of a real number is a balanced combination of the following two features:
\begin{quote}
\begin{enumerate}
\item[(i)] The size of the approximation (or its simplicity).
\item[(ii)] The accuracy of the approximation.
\end{enumerate}
\end{quote}    
The accuracy of an approximation of a given real number $x$ (say $x \approx \alpha$) is characterized by its error, which is equal to $|x - \alpha|$. The size of an approximation can also be naturally defined once a \emph{model of approximation} is fixed (see below). For example, for the \emph{rational model}, the size of the approximation $x \approx \frac{p}{q}$ (where $p , q \in \Z$, $q \neq 0$) can be defined by $\s(x \approx \frac{p}{q}) = |p q|$. What remains to be made precise is how these two features (i) and (ii) are combined. The objective of this paper is to make this combination rigorous. We then define a \emph{measure of intelligence} of an approximation of a given real number (depending on a model) such that: 
\begin{itemize}
\item If it is $< 1$, the approximation is naive;
\item If it is $\geq 1$, the approximation is intelligent.
\end{itemize}
In addition, more the \emph{measure of intelligence} of an approximation is greater more that approximation is intelligent.

\section{Models of approximation}
\begin{defi}
A \emph{model of approximation} is a map $\M : A \subset \Z^{*n} \rightarrow \R$ (where $n$ is a positive integer). The integer inputs of $\M$ are called its \emph{parameters} and we say that $\M$ is a model of $n$ parameters.
\end{defi}

\begin{expls}
The following expressions:
$$
\begin{array}{rl}
\displaystyle \frac{a}{b} ~, & \hspace*{4cm} (\M_1) \\[4mm]
\displaystyle a + b \sqrt{2} ~, & \hspace*{4cm} (\M_2) \\[4mm]
\displaystyle a + b \sqrt{c} ~, & \hspace*{4cm} (\M_3) \\[4mm]
\displaystyle \frac{a}{b + c \log{2} + d \log{3}} & \hspace*{4cm} (\M_4)
\end{array}
$$~\\[-3mm]
(where $a , b , c , d \in \Z^*$) define models of approximation. We refer to $\M_1$ as the \emph{rational model}; it is a model with two parameters. The model $\M_2$ is also with two parameters; while $\M_3$ is a model with 3 parameters and $\M_4$ is a model of 4 parameters.
\end{expls}

\begin{defi}
Let $\M$ be a model of approximation of $n$ parameters ($n \in \N$) and let $x$ be a real number. We say that $x$ is \emph{representable} in $\M$ if $x$ belongs to the image of $\M$. Furthermore, we say that an approximation of $x$ \emph{lies} in $\M$ if it has the form: $x \approx \M(a_1 , a_2 , \dots , a_n)$ (where $a_1 , a_2 , \dots , a_n \in \Z^*$).
\end{defi}

In the above examples, we observe that any approximation lying in $\M_2$ also lies in $\M_3$. In this case, we say that the model $\M_3$ is finer than the model $\M_2$; or equivalently, the model $\M_2$ is coarser than the model $\M_3$. More precisely, we have the following definition:

\begin{defi}
Let $\M$ and $\M'$ be two models of approximation and let $n$ and $n'$ be respectively their number of parameters. We say that $\M'$ is finer than $\M$ (or equivalently $\M$ is coarser than $\M'$) if the image of $\M$ is contained in the image of $\M'$.
\end{defi}

\section[Size and logarithmic size of an approximation]{The size and the logarithmic size of an approximation lying in a given model}

\begin{defi}
Let $\M$ be a model of approximation of $n$ parameters ($n \in \N$) and let $x$ be a real number. We define \emph{the size} of an approximation $x \approx \M(a_1 , a_2 , \dots , a_n)$ of $x$ in the model $\M$ by:
$$
\s\left(x \approx \M(a_1 , a_2 , \dots , a_n)\right) ~:=~ \left\vert a_1\right\vert \times \left\vert a_2\right\vert \cdots \times \left\vert a_n\right\vert .
$$
Accordingly, we define the \emph{logarithmic size} of the same approximation by:
$$
\s_{\log}\left(x \approx \M(a_1 , a_2 , \dots , a_n)\right) := \log\s\left(x \approx \M(a_1 , a_2 , \dots , a_n)\right) = \log{|a_1|} + \log{|a_2|} + \dots + \log{|a_n|} .
$$
Observe that $\s\left(x \approx \M(a_1 , a_2 , \dots , a_n)\right) \in \N$ and $\s_{\log}\left(x \approx \M(a_1 , a_2 , \dots , a_n)\right) \in [0 , + \infty)$. \\[1mm]
--- If $\s_{\log}\left(x \approx \M(a_1 , a_2 , \dots , a_n)\right) \neq 0$ (i.e., if not all parameters $a_i$ satisfy $a_i = \pm 1$), we say that the approximation $x \approx \M(a_1 , a_2 ,  \dots , a_n)$ of $x$ is \emph{admissible}.
\end{defi}

\section[Measure of intelligence of an approximation]{A measure of intelligence of an approximation lying in a given model}

Let $\M$ be a model of approximation of $n$ parameters ($n \geq 1$) and let $x$ be a nonzero real number not representable in $\M$ (i.e., $x \not\in \M(\Z^{*n})$). Let
\begin{equation}\label{eq0}
x ~\approx~ \M(a_1 , a_2 , \dots , a_n) \tag{$*$}
\end{equation}
(where $a_1 , a_2 , \dots , a_n \in \Z^*$) be an admissible approximation of $x$ lying in $\M$. We now define the measure of intelligence of \eqref{eq0} and determine the criterion that makes it intelligent. To do so, we argue on the digits of real numbers in some basis of numeration. We work in the decimal numeration system, but we will see shortly that our result is independent of this choice. The key idea behind our definition of the intelligence of an approximation is as follows:
\begin{quote}
The approximation \eqref{eq0} is intelligent if each digit of each of the integers $a_1 , a_2 , \dots , a_n$ contributes to producing at least one correct digit of $x$ (before or after the decimal point) in the approximation \eqref{eq0}. In other words, the approximation \eqref{eq0} is intelligent if the total number of digits of $a_1 , a_2 , \dots , a_n$ does not exceed the number of correct digits of $x$ (before and after the decimal point) that the approximation \eqref{eq0} gives. \\
If \eqref{eq0} is not intelligent, we say that it is naive.
\end{quote}

Since the number of digits of a positive integer $a$ is approximately\footnotemark[1] equal to $\frac{\log a}{\log{10}}$, \linebreak then the sum of the number of digits of the integers $a_1 , a_2 , \dots , a_n$ is approximately equal to
$$
\sum_{i = 1}^{n} \frac{\log{|a_i|}}{\log{10}} ~=~ \frac{\log|a_1 a_2 \cdots a_n|}{\log{10}} .
$$
Note that, up to a multiplicative constant, this quantity is precisely the logarithmic size of the approximation \eqref{eq0}.

Similarly, the number of digits before the decimal point of $x$ (i.e., the number of digits of the integer part of $|x|$) is approximately\footnotemark[1] equal to
$$
\frac{\log{|x|}}{\log{10}} .
$$ 
Further, the number of the correct digits after the decimal point that the approximation \eqref{eq0} of $x$ gives is approximately\footnotemark[1] equal to
\footnotetext[1]{In fact, these approximations are exact if we adopt the heuristic of a non-integer (resp. non-positive) number of digits}
$$
- \, \frac{\log{|x - \M(a_1 , a_2 , \dots , a_n)|}}{\log{10}} .
$$
Thus, according to our criterion above, the approximation \eqref{eq0} is intelligent if and only if we have:
$$
\frac{\log|a_1 a_2 \cdots a_n|}{\log{10}} ~\leq~ \frac{\log{|x|}}{\log{10}} - \frac{\log{|x - \M(a_1 , a_2 , \dots , a_n)|}}{\log{10}} ;
$$
that is, if and only if
$$
\frac{\log{|x|} - \log{|x - \M(a_1 , a_2 , \dots , a_n)|}}{\log|a_1 a_2 \cdots a_n|} ~\geq~ 1 .
$$
We note in passing that this condition is independent of the choice of the decimal base. This motivates the following definition:

\begin{defi}
Let $\M$ be a model of approximation of $n$ parameters ($n \geq 1$) and let $x$ be a nonzero real number not representable in $\M$ (i.e., $x \not\in \M(\Z^{*n})$). Let also
\begin{equation}
x \approx \M(a_1 , a_2 , \dots , a_n) \tag{$*$}
\end{equation} 
(where $a_1 , a_2 , \dots , a_n \in \Z^*$) be an admissible approximation of $x$ in the model $\M$. \\
We define the measure of intelligence of \eqref{eq0}, denoted $\mu\eqref{eq0}$, by:
$$
\boxed{
\mu\eqref{eq0} ~:=~ \frac{\log{|x|} - \log{|x - \M(a_1 , a_2 , \dots , a_n)|}}{\log|a_1 a_2 \cdots a_n|}
}~.
$$
We say that \eqref{eq0} is \emph{intelligent} if $\mu\eqref{eq0} \geq 1$ and we say that it is \emph{naive} in the contrary.
\end{defi}

\begin{rmk}
It is clear from the above that the larger the measure of intelligence of an approximation, the more intelligent that approximation is. Thanks to this measure $\mu$, we can now compare two approximations of the same real number $x$ in the same model and determine which is more intelligent, without relying on intuition.
\end{rmk}

\section{Examples}
In what follows, we provide several examples of calculation of measures of intelligence of approximations. Note that each approximation is expressed in the model that is simultaneously the finest possible and involves the fewest possible parameters. For example, the approximation $\mathbf{e} \approx 5 \sqrt{2} - 3 \sqrt{3} + 13 \sqrt{6} - 31$ of $\mathbf{e}$ is taken in the model $a \sqrt{b} + c \sqrt{d} + f \sqrt{g} + h$; thus neither in the model $a \sqrt{2} + b \sqrt{3} + c \sqrt{6} + d$  (because it is coarser than the previous one) nor in the model $\frac{a \sqrt{b} + c \sqrt{d} + f \sqrt{g} + h}{k}$ (because the parameter $k$ is supplementary).

\subsection{Approximations concerning the number ${\boldsymbol \pi}$}
\begin{enumerate}
\item Archimedes' approximation of $\pi$ is:
\begin{equation}\label{eq1}
\pi ~\approx~ \frac{22}{7} .
\end{equation}
The error of \eqref{eq1} is $|\pi - \frac{22}{7}| \approx 1.2 \times 10^{-3}$ and its measure of intelligence (in the rational model) is:
$$
\mu(\ref{eq1}) = \frac{\log{\pi} - \log{|\pi - \frac{22}{7}|}}{\log(22 \times 7)} = 1.55\dots \geq 1 .
$$ 
This shows that Archimedes' approximation of $\pi$ is intelligent.
\item The well-known approximation:
\begin{equation}\label{eq2}
\pi ~\approx~ \frac{355}{113}
\end{equation}
has the error $|\pi - \frac{355}{113}| \approx 2.6 \times 10^{-7}$ and has the measure of intelligence (in the rational model):
$$
\mu(\ref{eq2}) = \frac{\log{\pi} - \log{|\pi - \frac{355}{113}|}}{\log(355 \times 113)} = 1.53\dots \geq 1 .
$$
This is thus an intelligent approximation, though slightly less so than Archimedes' approximation above (despite being more accurate).

Note that both of these approximations of $\pi$ are regular continued fraction convergents of $\pi$. In the next section, we will prove that any regular continued fraction convergent of any given nonzero real number is an intelligent approximation of that number in the rational model.
\item The well-known approximation:
\begin{equation}\label{eq3}
\pi ~\approx~ \sqrt{2} + \sqrt{3}
\end{equation}
has the error $|\pi - (\sqrt{2} + \sqrt3)| \approx 4.6 \times 10^{-3}$ and has the measure of intelligence (in the model $\sqrt{a} + \sqrt{b}$):
$$
\mu(\ref{eq3}) = \frac{\log{\pi} - \log{|\pi - (\sqrt{2} + \sqrt{3})|}}{\log(2 \times 3)} = 3.63\dots \geq 1 .
$$
This shows that \eqref{eq3} is intelligent.
\item The well-known approximation:
\begin{equation}\label{eq4}
\pi ~\approx~ 2 \sqrt{\sqrt{6}}
\end{equation}
has the error $|\pi - 2 \sqrt{\sqrt{6}}| \approx 1.1 \times 10^{-2}$ and has the measure of intelligence (in the model $a \sqrt{\sqrt{b}}$):
$$
\mu(\ref{eq4}) = \frac{\log{\pi} - \log{|\pi - 2 \sqrt{\sqrt{6}}|}}{\log(2 \times 6)} = 2.26\dots \geq 1 .
$$
So \eqref{eq4} is intelligent.
\item The well-known approximation:
\begin{equation}\label{eq5}
\pi ~\approx~ \frac{20}{9} \sqrt{2}
\end{equation}
has the error $|\pi - \frac{20}{9} \sqrt{2}| \approx 1.1 \times 10^{-3}$ and has the measure of intelligence (in the model $\frac{a}{b} \sqrt{c}$):
$$
\mu(\ref{eq5}) = \frac{\log{\pi} - \log{|\pi - \frac{20}{9} \sqrt2|}}{\log(20 \times 9 \times 2)} = 1.35\dots \geq 1 .
$$
So \eqref{eq5} is intelligent.
\item The well-known approximation:
\begin{equation}\label{eq6}
\pi ~\approx~ \frac{20}{11} \sqrt{3}
\end{equation}
has the error $|\pi - \frac{20}{11} \sqrt{3}| \approx 7.6 \times 10^{-3}$ and has the measure of intelligence (in the model $\frac{a}{b} \sqrt{c}$):
$$
\mu(\ref{eq6}) = \frac{\log{\pi} - \log{|\pi - \frac{20}{11} \sqrt{3}|}}{\log(20 \times 11 \times 3)} = 0.92\dots < 1 .
$$
Hence \eqref{eq6} is naive, though nearly intelligent, as its measure of intelligence is close to $1$.
\item Kocha\'nski's approximation \cite{koch}
\begin{equation}\label{eq7}
\pi ~\approx~ \sqrt{\frac{40}{3} - 2 \sqrt{3}}
\end{equation}
has the error $\left|\pi - \sqrt{\frac{40}{3} - 2 \sqrt{3}}\right| \approx 5.9 \times 10^{-5}$ and has the measure of intelligence (in the model $\sqrt{\frac{a}{b} + c \sqrt{d}}$):
$$
\mu\eqref{eq7} = \frac{\log{\pi} - \log{\left|\pi - \sqrt{\frac{40}{3} - 2 \sqrt{3}}\right|}}{\log(40 \times 3 \times 2 \times 3)} = 1.65\dots \geq 1 .
$$
So \eqref{eq7} is intelligent.
\item Ramanujan's approximation
\begin{equation}\label{eq8}
\pi ~\approx~ \frac{3}{5} \left(3 + \sqrt{5}\right)
\end{equation}
has the error $\left|\pi - \frac{3}{5} (3 + \sqrt{5})\right| \approx 4.8 \times 10^{-5}$ and has the measure of intelligence (in the model $\frac{a}{b} (c + \sqrt{d})$):
$$
\mu\eqref{eq8} = \frac{\log{\pi} - \log{\left|\pi - \frac{3}{5} (3 + \sqrt{5})\right|}}{\log(3 \times 5 \times 3 \times 5)} = 2.04\dots \geq 1 .
$$
So \eqref{eq8} is intelligent.
\end{enumerate}
The remaining approximations concerning the number $\pi$ are all due to the author and they are all intelligent:
\begin{enumerate}
\item[9.] The approximation
\begin{equation}\label{eq9}
\pi ~\approx~ \sqrt{6 (\sqrt{7} - 1)}
\end{equation}
has the error $\approx 7.8 \times 10^{-4}$ and has the measure of intelligence $\approx 2.22$. It is therefore an intelligent approximation of $\pi$.
\item[10.] The approximation
\begin{equation}\label{eq10}
\pi ~\approx~ 3 + \frac{1}{\sqrt{65} - 1}
\end{equation}
has the error $\approx 5.1 \times 10^{-6}$ and has the measure of intelligence $\approx 2.53$. It is therefore an intelligent approximation of $\pi$.
\item[11.] The approximation
\begin{equation}\label{eq11}
\pi ~\approx~ 3 + \frac{\sqrt{30} - 1}{10 \sqrt{10}}
\end{equation}
has the error $\approx 10^{-5}$ and has the measure of intelligence $\approx 1.39$, confirming that it is intelligent. 
\item[12.] The approximation
\begin{equation}\label{eq12}
\pi ~\approx~ \frac{17}{11} \left(\frac{4}{\sqrt{15}} + 1\right)
\end{equation}
has the error $\approx 4.8 \times 10^{-7}$ and has the measure of intelligence $\approx 1.68$, confirming that it is intelligent.
\end{enumerate}

\subsection{Approximations concerning the number ${\boldsymbol e}$}
\begin{enumerate}
\item The well-known rational approximation
\begin{equation}
e ~\approx~ \frac{19}{7}
\end{equation}
has the error $\approx 4 \times 10^{-3}$ and has the measure of intelligence $\approx 1.33$. So it is an intelligent approximation of $e$. Actually, the fraction $\frac{19}{7}$ is a convergent of the continued fraction expansion of $e$.
\end{enumerate}
The remaining approximations concerning the number $e$ are all due to the author and they are all intelligent:
\begin{enumerate}
\item[2.] The approximation
\begin{equation}\label{eq14}
e ~\approx~ 3 - \frac{1}{3} \sqrt{\frac{5}{7}}
\end{equation}
has the error $\approx 8.6 \times 10^{-8}$ and has the measure of intelligence $\approx 3$. So it is a ``very'' intelligent approximation of $e$. This approximation is remarkable, as it is rare to achieve such a level of accuracy while maintaining so high a measure of intelligence.
\item[3.] The approximation
$$
e ~\approx~ \frac{8}{3} + \frac{1}{11} \left(\frac{5}{2 \sqrt{2}} - \frac{6}{5}\right)
$$
has the error $\approx 1.6 \times 10^{-8}$ and has the measure of intelligence $\approx 1.58$. So it is an intelligent approximation of $e$.
\item[4.] The approximation
$$
e ~\approx~ \sqrt{4 \sqrt{2} + \sqrt{3}}
$$
has the error $\approx 2.8 \times 10^{-5}$ and has the measure of intelligence $\approx 3.6$. So it is a ``very'' intelligent approximation of $e$. It should be noted that the high measure of intelligence of this approximation owes more to its simplicity than to its accuracy.
\item[5.] The approximation
$$
e ~\approx~ \frac{35 - \sqrt{26}}{11}
$$ 
has the error $\approx 1.1 \times 10^{-5}$ and has the measure of intelligence $\approx 1.35$. So it is an intelligent approximation of $e$.
\item[6.] The approximation
$$
e ~\approx~ \frac{8 \sqrt{3} - 2 \sqrt{2} + 8}{7}
$$
has the error $\approx 9.3 \times 10^{-7}$ and has the measure of intelligence $\approx 1.73$. So it is an intelligent approximation of $e$.
\item[7.] The approximation
$$
e ~\approx~ 5 \sqrt{2} - 3 \sqrt{3} + 13 \sqrt{6} - 31
$$
has the error $\approx 2.2 \times 10^{-7}$ and has the measure of intelligence $\approx 1.33$. So it is an intelligent approximation of $e$.
\item[8.] The approximation
$$
e ~\approx~ \frac{3 \sqrt{5} - 2}{\sqrt{3}}
$$
has the error $\approx 9.8 \times 10^{-7}$ and has the measure of intelligence $\approx 3.3$. So it is a ``very'' intelligent approximation of $e$.
\end{enumerate}

\subsection{Other notable approximations}

The table below lists further intelligent approximations of various real numbers, along with their errors and measures of intelligence.

\bigskip

\renewcommand{\arraystretch}{2.2}
\noindent\begin{tabular}{|c||c|c|c|}
\hline
\textbf{Real number} & \textbf{Approximation} & \textbf{Error} & \textbf{Measure of intelligence} \\
\hline
$\sqrt{e}$ & $\displaystyle\sqrt{3} - \frac{1}{12}$ & $3.8 \times 10^{-6}$ & $3.6$ \\
\hline
$\sqrt{\pi}$ & $\displaystyle\frac{13 \sqrt{7} - 14 \sqrt{5}}{4} + 1$ & $1.1 \times 10^{-8}$ & $1.86$ \\
\hline
$\pi$ & $\displaystyle 1 + \sqrt{\frac{8}{5} \sqrt{2} - \sqrt{3} - 4 \sqrt{5} + 13}$ & $2.6 \times 10^{-8}$ & $1.68$ \\
\hline
$\dfrac{e}{\pi}$ & $\displaystyle \frac{11}{5} \left(\sqrt{2} + \sqrt{7} - \frac{11}{3}\right)$ & $7.5 \times 10^{-8}$ & $1.6$ \\
\hline
$\sqrt{e^2 + \pi^2}$ & $\displaystyle 4 + \frac{119}{11880} + \frac{\sqrt{3}}{12}$ & $9.3 \times 10^{-13}$ & $1.52$ \\
\hline
\end{tabular}

\bigskip

\begin{rmks}~
\begin{enumerate}
\item Approximation \eqref{eq10} can be derived from the regular continued fraction expansion of $\pi$, namely $[3 ; 7 , 15 , 1 , 292 , 1 , \dots]$. Indeed, from this expansion, we have:
\[
\pi \approx [3 ; 7 , 15 , 1] = [3 ; 7 , 16] \approx [3 ; 7 , \overline{16}] .
\]
But $x := [3 ; 7 , \overline{16}]$ is a positive quadratic number whose explicit form can be easily derived. Putting $y := \frac{1}{x - 3} + 9$, we have $y = [\overline{16}]$, so $y$ satisfies the equation $y = 16 + \frac{1}{y}$. Solving this equation (taking the positive root, since $y > 16$), we obtain: $y = 8 + \sqrt{65}$, which then gives $x = 3 + \frac{1}{\sqrt{65} - 1}$. Consequently, we obtain the required approximation $\pi \approx 3 + \frac{1}{\sqrt{65} - 1}$. 

On the other hand, by remarking that $65 = 8^2 + 1^2$, the approximation \eqref{eq10} can be used to establish an easy geometric construction of a line of length ``very'' close to $\pi$ by using only a ruler and a compass. This should improve Kocha\'nski's famous geometric construction \cite{koch} invented for the same purpose. 
\item Similarly to the above point, the approximation \eqref{eq14} can also be found by using the regular continued fraction expansion of $e$, which is $e = [2 ; 1 , 2 , 1 , 1 , 4 , 1 , \dots , 1 , 2 n , 1 , \dots]$. Indeed, from this expansion, we have:
$$
e ~\approx~ [2 ; 1 , 2 , 1 , 1 , \overline{4 , 1 , 1 , 6 , 1 , 1}] .
$$
The calculations show that $[2 ; 1 , 2 , 1 , 1 , \overline{4 , 1 , 1 , 6 , 1 , 1}] = 3 - \frac{1}{3} \sqrt{\frac{5}{7}}$, which gives the required approximation.
\item Applying an algebraic transformation to a given approximation generally changes its measure of intelligence. So, if we apply such a transformation, we can pass from an intelligent approximation to another that is less so or even naive (and vice versa). For example, the approximation \eqref{eq14}: $e \approx 3 - \frac{1}{3} \sqrt{\frac{5}{7}}$, which has a measure of intelligence $\approx 3$, can be algebraically transformed into 
$$
e ~\approx~ 3 - \frac{\sqrt{35}}{21} ;
$$ 
which has a measure of intelligence (in the model $a + \frac{\sqrt{b}}{c}$) $\approx 2.24$, which is well lower than that of the previous one. For a rational approximation, it is clear that the measure of intelligence is maximized when the fraction is irreducible. Thus, the algebraic form in which an approximation is expressed plays a crucial role. 
\end{enumerate}
\end{rmks}

\section{Intelligent rational approximations} 

In this section, we present a detailed study of the intelligent approximations of a given irrational number in the rational model. As we will see below, this study is ultimately connected with the Diophantine approximation theory in $\R$. We shall prove that in the rational model, there are always infinitely many intelligent approximations of any given irrational number $x$. Particularly, we prove that any regular continued fraction convergent of $x$ is an intelligent approximation of $x$. Furthermore, we show that for certain $x \in \R \setminus \Q$, there exist intelligent rational approximations of $x$ that are not regular continued fraction convergents.

\subsection{Main results}

For ease of exposition and without loss of generality, we restrict our attention to positive real numbers in some of the results that follow. We begin with a straightforward proposition.

\begin{prop}\label{pp1}
Let $\alpha$ be an irrational number and let $p , q \in \Z^*$ such that $|p q| \neq 1$. Then the rational  approximation $\alpha \approx \frac{p}{q}$ is intelligent (in the rational model) if and only if one of the two following equivalent inequalities holds:
\begin{align}
\left\vert \alpha - \frac{p}{q}\right\vert & \leq \frac{|\alpha|}{|p q|} , \label{eqq1} \\[2mm]
\left\vert \frac{1}{\alpha} - \frac{q}{p}\right\vert & \leq \frac{1}{p^2} . \label{eqq2}
\end{align} 
\end{prop}

\begin{proof}
This is an immediate consequence of the definition of an intelligent approximation of a real number in a given model.
\end{proof}

From Proposition \ref{pp1}, we deduce the following important corollary:

\begin{coll}\label{cc1}
Let $\alpha$ be an irrational number. Then there are infinitely many intelligent rational approximations of $\alpha$.
\end{coll}

\begin{proof}
By Dirichlet's Diophantine approximation theorem (see \cite[Chapter 1]{bug}), there exist infinitely many rational numbers $\frac{q}{p}$ ($p , q \in \Z^*$) such that:
$$
\left\vert\frac{1}{\alpha} - \frac{q}{p}\right\vert ~\leq~ \frac{1}{p^2} .
$$ 
It follows (according to Proposition \ref{pp1}) that for each of such rational numbers $\frac{q}{p}$, the rational approximation $\alpha \approx \frac{p}{q}$ is intelligent. This confirms the required result of the corollary. 
\end{proof}

We have also the following effective corollary:

\begin{coll}\label{cc2}
Let $\alpha$ be a positive irrational number. Then any regular continued fraction convergent of $\alpha$ is an intelligent approximation of $\alpha$ in the rational model.
\end{coll}

\begin{proof}
If $[a_0 ; a_1 , a_2 , \dots]$ is the regular continued fraction expansion of $\alpha$, then the regular continued fraction expansion of $\frac{1}{\alpha}$ is clearly equal to
$$
\begin{cases}
[a_1 ; a_2 , a_3 , \dots] & \text{if } a_0 = 0 ,\\
[0 ; a_0 , a_1 , \dots] & \text{if } a_0 \neq 0 .
\end{cases}
$$
It follows from this fact that for any given regular continued fraction convergent $\frac{p}{q}$ of $\alpha$, the fraction $\frac{q}{p}$ is a regular continued fraction convergent of $\frac{1}{\alpha}$. So, according to the well-known properties of the regular continued fraction convergents of a real number (see \cite[Chapter 1]{khi}), we have:
$$
\left\vert\frac{1}{\alpha} - \frac{q}{p}\right\vert ~\leq~ \frac{1}{p^2} ;
$$
which concludes (according to Proposition \ref{pp1}) that the approximation $\alpha \approx \frac{p}{q}$ is intelligent in the rational model. The corollary is proved.
\end{proof}

\subsubsection*{A variant of the measure of intelligence $\mu$}
Proposition \ref{pp1} leads us to introduce a new measure of intelligence $\mu'$ which is quite close to $\mu$ and which is, perhaps, more practical and more significant for the rational model. 

\begin{defi}
Let $\M$ be a model of approximation of $n$ parameters ($n \in \N$) and $x$ be a nonzero real number not representable in $\M$. We define the measure of intelligence $\mu'$ of an admissible approximation $x \approx \M(a_1 , a_2 , \dots ,a_n)$ (where $a_1 , a_2 , \dots , a_n \in \Z^*$) of $x$ in the model $\M$ by:
$$
\mu'\left(x \approx \M(a_1 , a_2 , \dots , a_n)\right) ~:=~ \frac{\log\vert\M(a_1 , a_2 , \dots , a_n)\vert - \log\left\vert x - \M(a_1 , a_2 , \dots , a_n)\right\vert}{\log|a_1 a_2 \cdots a_n|} .
$$
If $\mu'\left(x \approx \M(a_1 , a_2 , \dots , a_n)\right) \geq 1$, we say that the approximation $x \approx \M(a_1 , a_2 , \dots , a_n)$ is $\mu'$-\emph{intelligent} and otherwise we say that it is $\mu'$-\emph{naive}.
\end{defi}

The following proposition can be easily checked.

\begin{prop}\label{pp3}
Let $\alpha$ be an irrational number and let $p , q \in \Z^*$ such that $|p q| \neq 1$. Then the approximation $\alpha \approx \frac{p}{q}$ is $\mu'$-intelligent in the rational model if and only if we have:
\begin{equation}
\left\vert\alpha - \frac{p}{q}\right\vert ~\leq~ \frac{1}{q^2} . \tag*{$\square$}
\end{equation}
\end{prop}

Using Dirichlet's Diophantine approximation theorem (see \cite[Chapter 1]{bug}) and the properties of the regular continued fraction convergents of a real number (see \cite[Chapter 1]{khi}), we immediately deduce (from Proposition \ref{pp3}) the following corollary:

\begin{coll}
Let $\alpha$ be an irrational number. Then there are infinitely many $\mu'$-intelligent rational approximations of $\alpha$. Particularly, any regular continued fraction convergent of $\alpha$ is a $\mu'$-intelligent approximation of $\alpha$ in the rational model. \hfill $\square$
\end{coll}

\textbf{For the rest of this paper, we discard $\boldsymbol{\mu'}$ and work exclusively with the original measure of intelligence $\boldsymbol{\mu}$}.

The following theorem shows the existence of intelligent rational approximations of some irrational numbers which are not regular continued fraction convergents of those numbers.

\begin{thm}\label{t1}
Let $x > 1$ be an irrational number and let $[a_0 ; a_1 , a_2 , \dots]$ be its regular continued fraction expansion (where $a_i \in \N$ for all $i \in \N_0$). Suppose that for some $n \in \N_0$, we have:
$$
a_{n + 1} ~\geq~ a_n - 1 ~\geq~ 1 .
$$ 
Then the rational approximation 
$$
x ~\approx~ [a_0 ; a_1 , a_2 , \dots , a_{n - 1} , a_n - 1]
$$
is intelligent, but it is not a regular continued fraction convergent of $x$. 
\end{thm}

\begin{proof}
It is obvious that $[a_0 ; a_1 , a_2 , \dots , a_{n - 1} , a_n - 1]$ is not a regular continued fraction convergent of $x$. So, it remains to show that the rational approximation $x \approx [a_0 ; a_1 , a_2 , \dots , a_{n - 1} , a_n - 1]$ is intelligent. To do so, we consider two cases. \\[1mm]
\textbullet{} \underline{If $n = 0$:} \\
In this case, we have to show that the approximation $x \approx a_0 - 1$ is intelligent. We have:
$$
x ~=~ a_0 + \frac{1}{a_1 + {}_{\ddots}} ~\leq~ a_0 + \frac{1}{a_1} ~\leq~ a_0 + \frac{1}{a_0 - 1}
$$ 
(because in this case, we have by hypothesis $a_1 \geq a_0 - 1$). It follows that:
$$
\vert x - (a_0 - 1)\vert ~\leq~ 1 + \frac{1}{a_0 - 1} ~=~ \frac{a_0}{a_0 - 1} ~<~ \frac{x}{a_0 - 1} ~~~~~~~~~~ \text{(because $x > a_0$)} ,
$$
which shows (according to Proposition \ref{pp1}) that $x \approx a_0 - 1$ is an intelligent approximation of $x$ (in the rational model). \\[1mm]
\textbullet{} \underline{If $n \geq 1$:} \\
Let $\frac{p}{q}$ and $\frac{p'}{q'}$ (where $p , q , p' , q' \in \N$) be the irreducible rational fractions defined by:
\begin{align*}
\frac{p}{q} & := [a_0 ; a_1 , a_2 , \dots , a_{n - 1}] , \\[2mm]
\frac{p'}{q'} & := [a_0 ; a_1 , a_2 , \dots , a_{n - 1} , a_n - 1] .
\end{align*} 
Setting
$$
y ~:=~ [0 ; 1 , a_{n + 1} , a_{n + 2} , \dots] ~=~ \frac{1}{1 + \frac{1}{a_{n + 1} + \frac{1}{a_{n + 2} + {}_{\ddots}}}} ,
$$
we can express $x$ in terms of $y$ as follows:
\begin{align*}
x & := [a_0 ; a_1 , a_2 , \dots] \\
& = [a_0 ; a_1 , a_2 , \dots , a_{n - 1} , [a_n ; a_{n + 1} , a_{n + 2} , \dots]] \\
& = \big[a_0 ; a_1 , a_2 , \dots , a_{n - 1} , a_n - 1 , [0 ; 1 , a_{n + 1} , a_{n + 2} , \dots]\big] \\
& = [a_0 ; a_1 , a_2 , \dots , a_{n - 1} , a_n - 1 , y] .
\end{align*}
It follows, according to the elementary properties of the regular continued fractions (see \linebreak \cite[Chapter 1]{khi}), that:
\begin{equation}\label{eq19}
x ~=~ \frac{p + p' y}{q + q' y} .
\end{equation}
Using this last formula, we have:
$$
\left\vert x - \frac{p'}{q'}\right\vert ~=~ \left\vert\frac{p + p' y}{q + q' y} - \frac{p'}{q'}\right\vert ~=~ \frac{|p q' - p' q|}{q' (q + q' y)} ~=~ \frac{1}{q' (q + q' y)}
$$
(because $|p q' - p' q| = 1$, since $\frac{p}{q}$ and $\frac{p'}{q'}$ are two consecutive convergents of a real number). Consequently, we have (according to \eqref{eq19}):
\begin{equation}\label{eq20}
\left\vert x - \frac{p'}{q'}\right\vert ~=~ \frac{x}{q' (p + p' y)} .
\end{equation}
On the other hand, from the definition of $y$ and the hypothesis of the theorem, we have:
\begin{equation}\label{eq21}
y ~>~ \frac{1}{1 + \frac{1}{a_{n + 1}}} ~\geq~ \frac{1}{1 + \frac{1}{a_n - 1}} ~=~ 1 - \frac{1}{a_n} .
\end{equation}
Next, by setting $\frac{p_{n - 2}}{q_{n - 2}}$ (where $p_{n - 2} , q_{n - 2} \in \N_0$) the irreducible rational fraction which is equal to $[a_0 , a_1 , \dots , a_{n - 2}]$ \big(with the convention $(p_{n - 2} , q_{n - 2}) = (1 , 0)$ if $n = 1$\big), we have (according to the elementary properties of the regular continued fractions):
$$
p' ~=~ (a_n - 1) p + p_{n - 2} ~\leq~ (a_n - 1) p + p ~=~ a_n p ~~~~~~~~~~ \text{(since $p_{n - 2} \leq p$)} .
$$  
Thus
$$
\frac{p}{p'} ~\geq~ \frac{1}{a_n} ~>~ 1 - y ~~~~~~~~~~ \text{(according to \eqref{eq21})} .
$$
Hence $\frac{p}{p'} + y > 1$; that is:
$$
p + p' y ~>~ p' .
$$
By inserting this in \eqref{eq20}, we finally obtain:
$$
\left\vert x - \frac{p'}{q'}\right\vert ~\leq~ \frac{x}{p' q'} ,
$$
which shows (according to Proposition \ref{pp1}) that the rational approximation $x \approx \frac{p'}{q'}$ \linebreak $= [a_0 , a_1 , \dots , a_{n - 1} ,$ $a_n - 1]$ is intelligent. This completes the proof of the theorem. 
\end{proof}

From the previous theorem, we deduce the following important corollary:

\begin{coll}
Let $x > 1$ be an irrational number whose regular continued fraction expansion contains a finite number of $1$'s. Then there exist infinitely many intelligent rational approximations of $x$ which are not regular continued fraction convergents of $x$. 
\end{coll}

\begin{proof}
Let $[a_0 ; a_1 , a_2 , \dots]$ be the regular continued fraction expansion of $x$. By hypothesis, we have $a_n \geq 2$ for any sufficiently large $n$. On the other hand, since ${(a_n)}_{n \in \N_0}$ is a sequence of positive integers, we have for infinitely many $n \in \N_0$: $a_{n + 1} ~\geq~ a_n - 1$. Consequently, we have for infinitely many $n \in \N_0$:
$$
a_{n + 1} ~\geq~ a_n - 1 ~\geq~ 1 .
$$
It follows from Theorem \ref{t1} that for any such $n$, the rational approximation $x \approx [a_0 ; a_1 , \dots ,$ $a_{n - 1} , a_n - 1]$ is intelligent but it is not a regular continued fraction convergent of $x$. This confirms the corollary.
\end{proof}

The following theorem provides another category of intelligent rational approximations of some irrational numbers that do not appear in their regular continued fraction convergents. 

\begin{thm}\label{t2}
Let $x > 1$ be an irrational number and let $[a_0 ; a_1 , a_2 , \dots]$ be its regular continued fraction expansion. Suppose that for some $n \in \N_0$, we have:
$$
2 ~\leq~ a_{n + 1} ~\leq~ a_n + 1 . 
$$
Then the rational approximation $x \approx [a_0 ; a_1 , a_2 , \dots , a_{n - 1} , a_n + 1]$ is intelligent, but it is not a regular continued fraction convergent of $x$. 
\end{thm}

\begin{proof}
Because $[a_0 ; a_1 , a_2 , \dots , a_{n - 1} , a_n + 1] = [a_0 ; a_1 , a_2 , \dots , a_{n - 1} , a_n , 1]$ and $a_{n + 1} \geq 2$, then $[a_0 ; a_1 , a_2 , \dots , a_{n - 1} , a_n + 1]$ is not a regular continued fraction convergent of $x$. Now, let us show that the rational approximation $x \approx [a_0 ; a_1 , a_2 , \dots , a_{n - 1} , a_n + 1]$ is intelligent. To do so, we distinguish two cases. \\[1mm]
\textbullet{} \underline{If $n = 0$:} \\
In this case, we have to show that the approximation $x \approx a_0 + 1$ is intelligent. We have:
$$
x ~=~ a_0 + \frac{1}{a_1 + {}_{\ddots}} ~\geq~ a_0 + \frac{1}{a_1 + 1} ~\geq~ a_0 + \frac{1}{a_0 + 2}
$$
(because in this case, we have by hypothesis $a_1 \leq a_0 + 1$). It follows that:
\begin{align*}
\vert x - (a_0 + 1)\vert ~=~ a_0 + 1 - x ~\leq~ (a_0 + 1) - \left(a_0 + \frac{1}{a_0 + 2}\right) & = \frac{a_0 + 1}{a_0 + 2} \\
& = \frac{1}{a_0 + 1} \left(a_0 + \frac{1}{a_0 + 2}\right) \\
& \leq \frac{x}{a_0 + 1} ,
\end{align*}
which shows (according to Proposition \ref{pp1}) that $x \approx a_0 + 1$ is an intelligent approximation of $x$ (in the rational model). \\[1mm]
\textbullet{} \underline{If $n \geq 1$:} \\
Let $\frac{p}{q}$ and $\frac{p'}{q'}$ (where $p , q , p' , q' \in \N$) be the irreducible rational fractions defined by:
\begin{align*}
\frac{p}{q} & := [a_0 ; a_1 , a_2 , \dots , a_{n - 1}] , \\[2mm]
\frac{p'}{q'} & := [a_0 ; a_1 , a_2 , \dots , a_{n - 1} , a_n + 1] 
\end{align*} 
and let
$$
y ~:=~ [0 ; -1 , a_{n + 1} , a_{n + 2} , \dots] ~=~ \frac{1}{-1 + \frac{1}{a_{n + 1} + \frac{1}{a_{n + 2} + {}_{\ddots}}}} 
$$
(remark that $y < -1$). We can express $x$ in terms of $y$ as follows:
\begin{align*}
x & := [a_0 ; a_1 , a_2 , \dots] \\
& = [a_0 ; a_1 , a_2 , \dots , a_{n - 1} , a_n + 1 , y] .
\end{align*}
It follows, according to the elementary properties of the regular continued fractions (see \linebreak \cite[Chapter 1]{khi}), that:
\begin{equation}\label{eqf1}
x ~=~ \frac{p + p' y}{q + q' y} .
\end{equation}
Using this last formula, we have:
$$
\left\vert x - \frac{p'}{q'}\right\vert ~=~ \left\vert\frac{p + p' y}{q + q' y} - \frac{p'}{q'}\right\vert ~=~ \frac{|p q' - p' q|}{q' |q + q' y|} ~=~ \frac{1}{q' |q + q' y|}
$$
(because $|p q' - p' q| = 1$, since $\frac{p}{q}$ and $\frac{p'}{q'}$ are two consecutive convergents of a real number). Consequently, we have (according to \eqref{eqf1}):
\begin{equation}\label{eqf2}
\left\vert x - \frac{p'}{q'}\right\vert ~=~ \frac{x}{q' |p + p' y|}
\end{equation}
On the other hand, from the definition of $y$ and the hypothesis of the theorem, we have:
$$
1 + \frac{1}{y} ~=~ \frac{1}{a_{n + 1} + {}_{\ddots}} ~>~ \frac{1}{a_{n + 1} + 1} ~\geq~ \frac{1}{a_n + 2} ,
$$
which gives:
\begin{equation}\label{eqf3}
y ~<~ -1 - \frac{1}{a_n + 1} .
\end{equation}
Next, by setting $\frac{p_{n - 2}}{q_{n - 2}}$ (where $p_{n - 2} , q_{n - 2} \in \N_0$) the irreducible rational fraction which is equal to $[a_0 ; a_1 , \dots , a_{n - 2}]$ \big(with the convention $(p_{n - 2} , q_{n - 2}) = (1 , 0)$ if $n = 1$\big), we have (according to the elementary properties of the regular continued fractions):
$$
p' ~=~ (a_n + 1) p + p_{n - 2} ~\geq~ (a_n + 1) p .
$$  
Thus
$$
\frac{p}{p'} ~\leq~ \frac{1}{a_n + 1} ~<~ - y - 1 ~~~~~~~~~~ \text{(according to \eqref{eqf3})} .
$$
Hence $\frac{p}{p'} + y < -1$; that is $p + p' y < - p'$. Thus
$$
|p + p' y| ~>~ p' .
$$
By inserting this in \eqref{eqf2}, we finally obtain:
$$
\left\vert x - \frac{p'}{q'}\right\vert ~<~ \frac{x}{p' q'} ,
$$
which shows (according to Proposition \ref{pp1}) that the rational approximation $x \approx \frac{p'}{q'}$ \linebreak $= [a_0 ; a_1 , \dots , a_{n - 1} ,$ $a_n + 1]$ is intelligent. This completes the proof of the theorem.
\end{proof}

\subsection{Applications of Theorems \ref{t1} and \ref{t2} to $\pi$, $e$, and various other real numbers}

\subsubsection*{For the number $\pi$}

The first applications of Theorem \ref{t1} for $\pi = [3 ; 7 , 15 , 1 , 292 , 1 , 1 , 1 , 2 , \dots]$ show that the two rational approximations $\pi \approx 2 ~(= [2])$ and $\pi \approx \frac{19}{6} ~(= [3 , 6])$ are both intelligent but neither of them is a regular continued fraction convergent of $\pi$. On the other hand, an application of Theorem \ref{t2} for $\pi$ shows that the rational approximation $\pi \approx \frac{521030}{165849} ~(=[3 ; 7 , 15 , 1 , 292 , 1 , 1 , 2])$ is intelligent but it is not a regular continued fraction convergent of $\pi$. \\
For the number $\pi$, we propose the following open problem:

\begin{leftbar}
\noindent\textbf{An open problem.} Is there a finite or an infinite number of intelligent rational approximations of $\pi$ that do not appear in its regular continued fraction convergents? 
\end{leftbar}

\subsubsection*{For the number $e$}
For the number $e = [2 ; 1 , 2 , 1 , 1 , 4 , 1 , \dots , 1 , 2 k , 1 , \dots]$, it is clear that Theorem \ref{t1} can be applied in only one way\footnote[2]{Note that the approximation $e \approx 1$ (obtained from the application of Theorem \ref{t1} for $n = 0$) is inadmissible because its logarithmic size is zero.} which gives that the rational approximation $e \approx \frac{5}{2} ~(= [2 , 1 , 1])$ is intelligent but it is not a regular continued fraction convergent of $e$. Similarly, Theorem \ref{t2} can also be applied in only one way and gives the same result (since $[2 , 2] = [2 , 1 , 1] = \frac{5}{2}$).

By direct computation, we find that the approximations $e \approx \frac{38}{14}$ and $e \approx \frac{386}{142}$ are also intelligent but they respectively reduce to the two regular continued fraction convergents $\frac{19}{7}$ and $\frac{193}{71}$ of $e$. This leads us to propose the following conjecture:
\begin{leftbar}
\noindent\textbf{Conjecture.} The approximation $e \approx \frac{5}{2}$ is the only intelligent rational approximation of $e$ which cannot be reduced to one of its regular continued fraction convergents.
\end{leftbar}

\subsubsection*{For the number $\sqrt{2}$}

The applications of Theorem \ref{t1} for $\sqrt{2} = [1 ; \overline{2}]$ show that for all $n \in \N_0$, the rational approximation $\sqrt{2} \approx [1 ; \underbrace{2 , 2 , \dots , 2}_{n \text{ times}} , 1]$ is intelligent but it is not a regular continued fraction convergent of $\sqrt{2}$. Actually, it is easy to show that for any positive integer $n$, we have:
$$
[1 ; \underbrace{2 , 2 , \dots , 2}_{n \text{ times}} , 1] ~=~ \frac{2}{[1 ; \underbrace{2 , 2 , \dots , 2}_{n \text{ times}}]} .
$$     
So for any regular continued fraction convergent $r$ of $\sqrt{2}$, the rational approximation $\sqrt{2} \approx \frac{2}{r}$ is intelligent but it is not a regular continued fraction convergent of $\sqrt{2}$. Consequently, the number $\sqrt{2}$ has infinitely many intelligent rational approximations outside its regular continued fraction convergents.

The applications of Theorem \ref{t2} for $\sqrt{2}$ essentially give the same results because we have for any $n \in \N_0$:
$$
[1 ; \underbrace{2 , 2 , \dots , 2}_{n \text{ times}} , 3] ~=~ [1 ; \underbrace{2 , 2 , \dots , 2}_{(n + 1) \text{ times}} , 1] .
$$ 

For the number $\sqrt{2}$, we propose the following conjecture:
\begin{leftbar}
\noindent\textbf{Conjecture.} Any intelligent rational approximation of $\sqrt{2}$ has one of the two forms: $r_n$ or $\frac{2}{r_n}$ ($n \in \N_0$), where $r_n$ denotes the $n$\textsuperscript{th} regular continued fraction convergent of $\sqrt{2}$.
\end{leftbar}

\subsubsection*{For the number $\sqrt{5}$} 

Let ${(F_n)}_{n \in \N_0}$ and ${(L_n)}_{n \in \N_0}$ denote respectively the usual Fibonacci and Lucas sequences, which are defined by:
\[
\begin{cases}
F_0 = 0 ~,~ F_1 = 1 , \\
F_{n + 2} = F_n + F_{n + 1} ~~ (\forall n \in \N_0)
\end{cases} \qquad \text{and} \qquad
\begin{cases}
L_0 = 2 ~,~ L_1 = 1 , \\
L_{n + 2} = L_n + L_{n + 1} ~~ (\forall n \in \N_0)
\end{cases} .
\] 
It is easy to show that we have $\mathrm{gcd}(F_n , L_n) = 2$ if $n \equiv 0 ~(\bmod~ 3)$ and $\mathrm{gcd}(F_n , L_n) = 1$ if $n \not\equiv 0 ~(\bmod~ 3)$. The known fact that $\lim_{n \rightarrow + \infty} \frac{L_n}{F_n} = \sqrt{5}$ allows us to consider the fractions $\frac{L_n}{F_n}$ ($n \geq 2$) as rational approximations of the number $\sqrt{5}$. Besides, those rational approximations are reducible if and only if $n \equiv 0 ~(\bmod~ 3)$. In that case, the fraction $\frac{L_n}{F_n}$ reduces to $\frac{L_n/2}{F_n/2}$ and then becomes simpler. In fact, the fractions $\frac{L_n/2}{F_n/2}$ for $n \equiv 0 ~(\bmod~ 3)$ (i.e., the fractions $\frac{L_{3 n}/2}{F_{3 n}/2}$ for $n \geq 1$) are precisely the regular continued fraction convergents of $\sqrt{5}$. On the other hand, by calculations, we can easily show that for any $n \in \N$, the approximation $\sqrt{5} \approx \frac{L_n}{F_n}$ is intelligent (in the rational model); so the rational approximations $\sqrt{5} \approx \frac{L_{3 n + 1}}{F_{3 n + 1}}$ ($n \geq 1$) and $\sqrt{5} \approx \frac{L_{3 n + 2}}{F_{3 n + 2}}$ ($n \geq 1$) are all intelligent but none of them is a regular continued fraction convergent\footnote[3]{We can show that the equations $\frac{L_{3 n + 1}}{F_{3 n + 1}} = \frac{L_{3 m}}{F_{3 m}}$ and $\frac{L_{3 n + 2}}{F_{3 n + 2}} = \frac{L_{3 m}}{F_{3 m}}$ are impossible for $n , m \in \N$.} of $\sqrt{5}$. Consequently, the number $\sqrt{5}$ has infinitely many intelligent rational approximations outside its regular continued fraction convergents. Notably, the same results can be recovered using Theorems \ref{t1} and \ref{t2}. Indeed, because we have $\sqrt{5} = [2 ; \overline{4}]$, Theorem \ref{t1} applies and shows that for any $n \in \N_0$, the rational approximation $\sqrt{5} \approx [2 ; \underbrace{4 , 4 ,\dots , 4}_{n \text{ times}} , 3] = \frac{L_{3 n + 4}}{F_{3 n + 4}}$ is intelligent but it is not a regular continued fraction convergent of $\sqrt{5}$. Also, Theorem \ref{t2} applies and shows that for any $n \in \N_0$, the rational approximation $\sqrt{5} \approx [2 ; \underbrace{4 , 4 , \dots , 4}_{n \text{ times}} , 5] = \frac{L_{3 n + 5}}{F_{3 n + 5}}$ is intelligent but it is not a regular continued fraction convergent of $\sqrt{5}$.

For the number $\sqrt{5}$, we propose the following conjecture:
\begin{leftbar}
\noindent\textbf{Conjecture:} Every intelligent rational approximation of $\sqrt{5}$ has one of the two forms: \linebreak $\frac{L_{3 n}/2}{F_{3 n}/2}$ or $\frac{L_n}{F_n}$ ($n \in \N$). 
\end{leftbar}        

\subsection{Irrational numbers with bounded or unbounded intelligence measure in the rational model}

Irrational numbers are of different natures in terms of the intelligence of their rational approximations. An important class of them does not possess ``very intelligent'' rational approximations! The result below shows that the irrational numbers for which the set of the measure of intelligence of their rational approximations is unbounded are exactly the Liouville numbers. First, let us recall the definition of the Liouville numbers:

\begin{defi}[see \cite{bug}]
An irrational number $x$ is called \emph{a Liouville number} if for every positive integer $n$, there exist integers $a , b$, with $b > 1$, such that:
$$
\left|x - \frac{a}{b}\right| ~<~ \dfrac{1}{b^n} .
$$
\end{defi} 

We have the following:

\begin{thm}\label{t3}
Let $x$ be an irrational number. Then the set of the measures of intelligence of all the (admissible) rational approximations of $x$ is unbounded if and only if $x$ is a Liouville number.
\end{thm}

\begin{proof}~\\
\textbullet{} Suppose that $x$ is a Liouville number. Then for any positive integer $n$, there exist $a_n , b_n \in \Z^*$ (with $b_n \geq 2$) such that:
\begin{equation}\label{eqq3}
\left|x - \frac{a_n}{b_n}\right| ~<~ \frac{1}{b_n^n} .
\end{equation}
From \eqref{eqq3}, we derive that for all $n > 1$, we have: $|\frac{a_n}{b_n}| < |x| + \frac{1}{b_n^n} < |x| + 1$ and then we have $\log|\frac{a_n}{b_n}| < \log(|x| + 1) \leq |x|$; hence
\begin{equation}\label{eqq4}
\log\left|\frac{a_n}{b_n}\right| ~<~ |x| .
\end{equation}
Next, using \eqref{eqq3} and \eqref{eqq4}, we have:
\begin{align*}
\mu\left(x \approx \frac{a_n}{b_n}\right) & = \frac{\log|x| - \log\left|x - \frac{a_n}{b_n}\right|}{\log|a_n b_n |} \\[2mm]
& > \frac{\log|x| + n \log(b_n)}{\log|a_n b_n|} ~~~~~~~~~~ \text{(according to \eqref{eqq3})} \\[2mm]
& = \frac{\log|x| + n \log(b_n)}{\log|\frac{a_n}{b_n}| + 2 \log(b_n)} \\[2mm]
& > \frac{\log|x| + n \log(b_n)}{|x| + 2 \log(b_n)} ~~~~~~~~~~ \text{(according to \eqref{eqq4})} \\[2mm]
& = \frac{n + \frac{\log|x|}{\log(b_n)}}{2 + \frac{|x|}{\log(b_n)}} \\[2mm]
& \geq \frac{n + \frac{\log|x|}{\log(b_n)}}{2 + \frac{|x|}{\log{2}}} ~~~~~~~~~~ \text{(since $b_n \geq 2$)} .
\end{align*}
But because ${(\frac{\log|x|}{\log(b_n)})}_n$ is a bounded sequence, we have $\lim_{n \rightarrow + \infty} \frac{n + \frac{\log|x|}{\log(b_n)}}{2 + \frac{|x|}{\log{2}}} = + \infty$. Thus
$$
\lim_{n \rightarrow + \infty} \mu\left(x \approx \frac{a_n}{b_n}\right) ~=~ + \infty ,
$$
concluding that the set of the measures of intelligence of all the (admissible) rational approximations of $x$ is unbounded. \\[1mm]
\textbullet{} Conversely, suppose that $x$ is not a Liouville number. Then there exists a positive integer $k$ such that for any rational approximation $\frac{a}{b}$ of $x$ (with $a, b \in \Z^*$ and $b \geq 2$), we have:
$$
\left|x - \frac{a}{b}\right| ~\geq~ \frac{1}{b^k} .
$$
This gives:
\begin{align*}
\mu\left(x \approx \frac{a}{b}\right) & = \frac{\log|x| - \log\left|x - \frac{a}{b}\right|}{\log|a b|} \\[2mm]
& \leq \frac{\log|x| + k \log{b}}{\log|a b|} \\[2mm]
& = \frac{\log|x|}{\log|a b|} + \frac{k \log{b}}{\log|a b|} \\[2mm]
& \leq \frac{|x|}{\log{2}} + k ,
\end{align*}
showing that any rational approximation of $x$, with denominator $\geq 2$, has a measure of intelligence bounded (above) by $(\frac{|x|}{\log{2}} + k)$. On the other hand, it is immediate that the measure of intelligence of any rational approximation of $x$, with denominator $1$ and numerator $\not\in \{0 , 1 , -1\}$ (i.e., any admissible integer approximation of $x$), is bounded above by $\frac{|x| - \log{d(x , \Z)}}{\log{2}}$ (where $d$ \linebreak denotes the usual distance on $\R$). This implies the required result and completes this proof.  
\end{proof}

We derive from Theorem \ref{t3} the following corollary:

\begin{coll}
Let $\alpha$ be an irrational algebraic number. Then the set of the measures of intelligence of all the {\rm(}admissible{\rm)} rational approximations of $\alpha$ is bounded.
\end{coll}

\begin{proof}
By Liouville's Diophantine approximation theorem (see \cite[§1]{bug}), $\alpha$ is not a Liouville number. The required result then immediately follows from Theorem \ref{t3}.
\end{proof}

\begin{rmk}
If an irrational real number $x$ is not a Liouville number then there exist $\kappa > 0$ and $c = c(x , \kappa) > 0$ such that for any rational number $\frac{a}{b}$ ($a , b \in \Z$, $b > 0$), we have:
$$
\left|x - \frac{a}{b}\right| ~\geq~ \frac{c(x , \kappa)}{b^{\kappa}} . 
$$
Such $\kappa$ (not necessarily integers) are called \emph{irrationality measures} of $x$. Besides, if $c(x , \kappa)$ can be explicitly computed, then $\kappa$ is called \emph{an effective irrationality measure} of $x$. \\
Many important transcendental numbers have been identified as non-Liouville numbers whose irrationality measures have been computed. Among these numbers, we list: $\pi , e , \log{2} , \log{3} , \zeta(3)$ (see \cite[§11.3, pp. 362-386]{bor}, \cite{mar}, \cite{rv}, \cite{sal1}, \cite{sal2}). Consequently, for each of these numbers, the set of the measures of intelligence of all its (admissible) rational approximations is bounded. In other words (roughly speaking):
\begin{leftbar}
\noindent None of the numbers $\pi , e , \log{2} , \log{3} , \zeta(3)$ admits a ``very intelligent'' rational approximation.
\end{leftbar}    
\end{rmk}

\section{An important open problem}

Given a model of approximation $\M$, let $\overline{\M}$ denote the closure of the image of $\M$ (with respect to the usual topology of $\R$). Numerical experiences suggest that for any $x \in \overline{\M} \setminus \M$, there exist intelligent approximations of $x$ in $\M$. However, proving or disproving this assertion seems difficult in general. The success achieved for the rational model (cf. above) relies on Dirichlet's approximation theorem, which itself follows from Dirichlet's pigeonhole principle (see \cite[§1]{bug}). Unfortunately, this elementary principle cannot be applied to other models. We therefore ask whether there exist other methods --based perhaps on density arguments in $\R$ or on the distribution modulo $1$ of real sequences-- that would allow one to settle the problem, even in particular cases.


\begin{thebibliography}{99}

\bibitem{bor}
{\sc J. M. Borwein} and {\sc P. B. Borwein}. {\it Pi and the AGM. A Study in Analytic Number Theory and Computational Complexity}, Wiley, New York, 1987. 

\bibitem{bug}
{\sc Y. Bugeaud}. {\it Approximation by algebraic numbers}, Cambridge Tracts in Mathematics, vol. {\bf 160}, Cambridge University Press, Cambridge, 2004.

\bibitem{khi}
{\sc A.Y. Khinchin}. {\it Continued fractions}, University of Chicago Press, 1964.

\bibitem{koch}
{\sc A. A. Kocha\'nski}. Observationes Cyclometricae ad facilitandam Praxin accomodatae, {\it Acta Eruditorum}, {\bf 4} (1685), p. 394-398.

\bibitem{mar}
{\sc R. Marcovecchio}. The Rhin-Viola Method for $\log{2}$, {\it Acta Arith}, {\bf 139} (2009), p. 147-184. 

\bibitem{rv}
{\sc G. Rhin \& C. Viola}. The Group Structure for $\zeta(3)$, {\it Acta Arith}, {\bf 97} (2001), p. 269-293.
\bibitem{sal1}

{\sc V. Kh. Salikhov}. On the Irrationality Measure of $\ln{3}$, {\it Dokl. Akad. Nauk}, {\bf 417} (2007), p. 753-755. Translation in {\it Dokl. Math}, {\bf 76} (2007), p. 955-957. 

\bibitem{sal2}
\idem. On the Irrationality Measure of pi, {\it Usp. Mat. Nauk}, {\bf 63} (2008), p. 163-164. English transl. in {\it Russ. Math. Surv}, {\bf 63} (2008), p. 570-572.

\end{thebibliography}
\end{document}